\documentclass[12pt,reqno]{amsart}
\usepackage[top=1in, bottom=1in, left=1in, right=1in]{geometry}
\pdfoutput=1 

\usepackage{amsmath}
\usepackage{amsthm}
\usepackage{amsfonts}
\usepackage{amssymb}

\usepackage{graphicx}
\usepackage{multirow}
\usepackage{color}
\usepackage[colorlinks=true]{hyperref}
\hypersetup{colorlinks=true, citecolor=green, linkcolor=green, filecolor=magenta, urlcolor=cyan}

\usepackage{array}

\usepackage{times}
\usepackage[utf8]{inputenc}
\usepackage[english]{babel}

\usepackage{booktabs}

\usepackage{rotating}
\usepackage{tabularx}
\usepackage{threeparttable}

\theoremstyle{plain}

\begin{document}

\title{On the number of prime numbers between $n^2$ and ${(n+1)}^2$}

\author{Jimin Li}
\email{1462228007@qq.com}

\author{Haonan Li}
\email{lihaonan@umich.edu}

\subjclass{11N05, 11R44, 11R45}
\keywords{Legendre’s conjecture, Euler’s totient function, coprime numbers, proportion of coprime numbers, proportion of backwards coprime numbers}

\begin{abstract}
Let $p_{r+1}-1>n \geq p_r-1$, based on a sequence $\{1,2,3\cdots\ M_r(M_r=p_1p_2\cdots p_r)\}$, we compare the density of coprime numbers and establish a correlation between the proportions of coprime numbers in the ranges from 1 to consecutive square numbers. Then, we derive the relationship between the number of coprimes in the interval of $n^2 \sim {(n+1)}^2$ and the proportion of coprimes in the interval of $1 \sim n^2$, proving that there is at least one prime number between any $n^2$ and ${(n+1)}^2$. By extending our research to the range of $1 \sim M_r^2$, we establish the relationship between the proportions of backwards coprime numbers in the ranges from ${M_r}^2$ to consecutive square numbers; furthermore, we establish a relationship between the proportions of coprimes in small interval and the whole interval. Then, in conclusion, the number of coprimes between $n^2$ and ${(n+1)}^2$ is greater than $n\prod_{i=1}^{r}{(1-\frac{1}{p_i}})$, thus proving that there are at least 2 prime numbers between $n^2$ and ${(n+1)}^2$.

\end{abstract}

%\medskip 
%\textbf{Keywords: Legendre’s conjecture, Euler’s totient function, coprime numbers, proportion of coprime numbers, proportion of backwards coprime numbers}
%\medskip 

\maketitle

\section{\textbf{Introduction}}\label{sec1}
Approximately 220 years ago, the French mathematician Adrian Marie Legendre (1752-1833) proposed the conjecture that there is at least one prime number between any two consecutive square numbers, which is also expressed as follows: For any positive integer n, there exists a prime number $p$ satisfying $n^2<p<{(n+1)}^2$. This is known as Legendre’s conjecture.
For hundreds of years, many mathematicians have devoted themselves to the study of prime numbers. Many have achieved this goal, which has promoted the continuous development of number theory. The common methods for studying prime numbers include the sieve method, the circle method, the density method \cite{rf1, rf2, rf3, rf4}.
Many mathematicians such as Luogeng Hua applied permutation and combination, an elementary mathematical method, to prove the Bertrand-Chebyshev theorem, which states that there is at least one prime number between $n$ and $2n$ \cite{rf5}.
However, in the past 20 years, sieve methods and improved sieve methods have become more popular among mathematicians. For example, mathematician Yitang Zhang proved the weakened form of the twin prime conjecture; that is, Zhang found that there are infinitely many prime pairs with a gap less than 70 million  \cite{rf6}.
James Maynard used an improved GPY sieve method to further reduce the maximum possible difference between two prime gaps \cite{rf7, rf8}  in a prime pair from 70 million to 600; K. Ford, B. Green, S. Konyagin, and T. Tao researched large gaps between consecutive prime numbers \cite{rf9, rf10}.
However, Legendre's conjecture, as a long-standing mathematical problem, is still being studied today. Legendre's conjecture involves the distribution of prime numbers, and its proof or falsification may help to further understand the nature and the characteristics of prime number distribution. This has important implications for mathematical research and other scientific fields related to mathematics, especially those utilizing prime numbers, such as cryptography, computer science, and secure communication.\par
Whether any prime numbers exist between $n^2$ and ${(n+1)}^2$ depends on the distribution of prime numbers from $1$ to $n^2$. Similarly, the proportion of coprime numbers between $n^2$ and ${(n+1)}^2$ is related to both the distribution of coprime numbers from $1$ to $M_r$ and to the distribution of coprime numbers from $1$ to ${M_r}^2$. Accordingly, Euler’s totient function and its extended applications are helpful for analysing these situations. \par
We adopt a new research method. First, we use coprime numbers associated with square numbers rather than pursue the direct study of the prime numbers between square numbers as has been the approach of previous researchers. Second, we adopt a method to compare the density of coprime numbers. By this method, we establish not only the relationships between the proportions of forward coprime numbers associated with consecutive square numbers, but also the relationships between the proportions of backwards coprime numbers associated with consecutive square numbers.\par
Moreover, we extend the scope of related research. Although the research object is only the number of prime numbers in the range of $n^2 \sim {(n+1)}^2$, the scope of the research is extended to the range of $1 \sim {M_r}^2$. Thus, we can establish the relationship between the proportions of coprimes in small intervals and whole intervals.\par
By establishing the relationship between the proportions of coprime numbers associated with consecutive square numbers, we derive the correlation between the number of coprimes in the interval of $n^2 \sim {(n+1)}^2$ and the proportion of coprime numbers associated square numbers. We let $p_{r+1}-1>n\geq p_r-\ 1$, thereby proving that the Legendre conjecture that there is at least one prime number between $n^2$ and $(n+1)^2$.\par
The relationships between the proportions of coprimes in small interval and the whole interval are established. We successfully derive that when $n\geq p_r-1$, the number of coprime numbers between $n^2$ and $\left(n+1\right)^2$ is greater than $n\prod_{i=1}^{r}{(1-\frac{1}{p_i}})$, thus proving that there are at least 2 prime numbers between $n^2$ and $(n+1)^2$ when $p_{r+1}>n+1>n \geq p_r-1$. \par
The remainder of this article is structured as follows: In Section 2, we summarize the preliminaries necessary for the derivations. In Section 3, we present the proof of Legendre’s conjecture based on the proportions of coprimes associated with consecutive square numbers. In Section 4 “Further research on prime numbers between $n^2$ and ${(n+1)}^2$”, by using theorems $\ref{theo3}$, $\ref{theo4}$, and $\ref{theo5}$, we establish the proportional relationship between the backwards coprime numbers of consecutive square numbers, and we derive the relationship between the proportion of coprime numbers in small interval and the proportion of coprime numbers in the overall interval, thus proving that there are at least two prime numbers between consecutive square numbers. The "Conclusion" section presents a brief summary of this article.

\section{\textbf{Preliminaries}}\label{sec2}
\subsection{Notation. \\} \label{subsec2.1}
\vskip2mm
$p,\ p_1,\ p_2\ \cdots p_i\cdots p_r, p_{r+1}, p_{r+2}\ldots p_m$: prime numbers. \\
$n, r, i, m, X, M$: positive integers.\\
$\pi(X)$: the number of prime numbers between $1$ and $X$ \\
$M_r: M_r=p_1 p_2 \cdots p_r$, where $p_1, p_2 \cdots p_r$ are all prime numbers not greater than $p_r$. \\
$A$: sequence of positive integers $\{1, 2, 3 \cdots M_r\}$ \\
$B$: sequence of all the numbers that are coprime with $M_r$ in sequence $A$, $\{1, p_{r+1}, p_{r+2} \ldots p_m\}$ \\
$C$: sequence of positive integers $\{1, 2, 3 \cdots {M_r}^2\}$ \\
$\varphi(M)$: Euler’s totient function, which is the number of coprime numbers that are coprime with $M$ for the positive integers less than or equal to $M$. \\
$\varphi(X, M_r)$: the number of coprime numbers that are coprime with $M_r$ among the positive integers less than or equal to $X$. This function is an extension of the Euler function and is specifically defined in this article. Our extended function inherits the characteristics of the Euler function and expands its scope of application, and plays a key role in the proof of the theorems in this article. \\
$f(X, M_r)$: the proportion function of coprime numbers, which represents the proportion of the numbers between $1$ and $X$ that are coprime with $M_r$ . This function is also called the proportion function of forward coprime numbers. \\
$f(\overline{X}, M_r)$: the proportion function of backwards coprime numbers, which represents the proportion of the numbers between ${M_r}^2$ and $X+1$ that are coprime with $M_r$. \\
$\ell_n$: the number of coprimes between $n^2$ and ${(n+1)}^2$.

\subsection{Euler’s totient function. \\}  \label{subsec2.2}
\quad According to the principle of the Euler function, the number of elements in sequence B is the output of Euler’s totient function $\varphi(M_r)$.\\
\begin{equation}
\begin{aligned}
\varphi(M_r)&=M_r-\sum\limits_{1\leq i\leq r}\!\frac{M_r}{p_i}+
                  \sum\limits_{1\leq i\leq j\leq r}\!\frac{M_r}{p_i p_j}-\cdots+
                  (-1)^k\!\sum\limits_{1\leq i\dots\leq k\leq r}\!\frac{M_r}{p_i\cdots p_k}  \\
             &\quad+(-1)^r\!\frac{M_r}{p_i\cdots p_r}   \\
&=M_r(1-\frac{1}{p_1})(1-\frac{1}{p_2})\cdots (1-\frac{1}{p_r})  \\
&=M_r\prod\limits_{i=1}^r(1-\frac{1}{p_i})  \\
&=(p_1-1)(p_2-1)\cdots (p_r-1) 
\end{aligned}
\end{equation}

Because $f(M_r)$ is the proportion of the coprime numbers in $A$, 
\begin{equation}
f(M_r)=\frac{\varphi(M_r)}{M_r}=\prod\limits_{i=1}^r(1-\frac{1}{p_i})
\end{equation}
\ \quad Let $M_{r-1}=p_1 p_2\cdots p_{r-1}$, then
\begin{equation}
\begin{aligned}
\varphi(M_{r-1})&=M_{r-1}\prod\limits_{i=1}^{r-1}(1-\frac{1}{p_i})\\
&=(p_1-1)(p_2-1)\cdots (p_{r-1}-1)
\end{aligned}
\end{equation}

\subsection{Extended applications of Euler’s totient function. \\}  \label{subsec2.3}
\quad Here, $\varphi(X,M_r)$ is a specially defined function that is an extension of Euler's function, this new function is crucial in this study. This function represents the number of coprimes that are coprime with $M_r$ among the positive integers less than or equal to $X$, where $X$ can be any positive integer. This function inherits the features of Euler's function and expands its scope of application. Our new function has the following characteristics.
\begin{align}
\varphi&(M_r,M_r)=\varphi(M_r) \\
\varphi&(M_r,M_{r-1})=p_r\varphi(M_{r-1})\\
\varphi&({p_r}^2,M_r)=\varphi({p_r}^2,M_{r-1})-2 \\
\varphi&({p_r}^2-1,M_r)=\varphi({p_r}^2-1,M_{r-1})-1 \\
\varphi&({(p_r-1)}^2,M_r)=\varphi({(p_r-1)}^2,M_{r-1})-1 \label{eq11}  \\
\varphi&(N,M_r)=1 \quad (when \, N\leq p_r) \\
\varphi&({M_r}^2,M_r)=M_r\varphi(M_r)  
\end{align}
\quad Because $\ell_n$ represents the number of coprimes between $n^2$ and ${(n+1)}^2$,
\begin{align}
\ell_n=&\varphi((n+1)^2,M_r)-\varphi(n^2,M_r)\\
  =&\{\varphi(M_r)-\varphi(n^2,M_r)\}-\{\varphi(M_r)-\varphi((n+1)^2,M_r)\} \notag
\end{align}
\quad When $p_{r+1}>n+1>n\geq\ p_r-1$, $\ell_n$ represents the number of prime numbers between $n^2$ and ${(n+1)}^2$.

\subsection{The proportion function of coprime numbers and the proportion function of backwards coprime numbers. \\}  \label{subsec2.4}
\quad $f(X,M_r)$ and $f(\overline{X},M_r)$ are functions specifically defined in this article. Given  the definitions of  both functions, we know that:
\begin{equation}
\begin{aligned}
f(X,M_r)&=\frac{\varphi(X,M_r)}{X} \\
f(\overline{X},M_r)&=\frac{\varphi({M_r}^2,M_r)-\varphi(X,M_r)}{{M_r}^2-X}
\end{aligned}
\end{equation}
\qquad Then we can obtain
\begin{align}
\ell_n=&(n+1)^2 f((n+1)^2,M_r)-n^2 f(n^2,M_r) \label{eqln}  \\
=&({M_r}^2-n^2) f(\overline{n^2},M_r)-[{M_r}^2-(n+1)^2] f(\overline{(n+1)^2},M_r) \notag
\end{align}
\vskip2mm

\section{\textbf{On coprime numbers between $n^2$ and $(n+1)^2$ that are coprime 
with $M_r$}} \label{sec3}
In this section, we establish the relationship between the proportions of coprimes associated with consecutive square numbers through Theorem $\ref{theo1}$, and we derive the relationship between the number of prime numbers and the proportion of coprimes associated with the square numbers through Theorem $\ref{theo2}$, thereby proving the Legendre conjecture.

\newcounter{counter1}
\setcounter{counter1}{0}
\newtheorem{mythm}[counter1]{Theorem}
\begin{mythm}\label{theo1}
General correlation model of the proportions of coprime numbers between 1 and consecutive square numbers:  
\begin{equation}
f((n+1)^2,M_r) > \frac{n}{n+1} \cdot f(n^2,M_r) \label{eqfn} %\label{eq48}
\end{equation}
\end{mythm}

\begin{proof}
We know that $\frac{\pi(X)}{X}$, which is the proportion of prime numbers between $1$ and $X$, gradually decreases as $X$ increases. This is an overall general trend, although the proportions of prime numbers for consecutive square numbers do not fully conform to this rule, for example, $\frac{\pi(7^2)}{7^2}=\frac{15}{49}=0.306122$, $\frac{\pi(6^2)}{6^2}=\frac{11}{36}=0.305556$, obviously, $\frac{\pi(7^2)}{7^2}>\frac{\pi(6^2)}{6^2}$. The same is true for the proportions of coprime numbers between 1 and consecutive square numbers. Is there any other definite relationship besides this? \par

For convenience, let $a_n=\varphi(n^2,M_r)$, and let $0<t<1$ \par
When $n\geq p_r-1$, we find that
\begin{equation}
\frac{a_n}{n(n+1)}<\frac{a_n}{n^2}<\frac{a_n+t\ell_n}{n^2} \notag 
\end{equation}
 \qquad and 
 \begin{equation}
\frac{a_n}{n(n+1)}<\frac{a_n+t\ell_n}{n(n+1)}<\frac{a_n+t\ell_n}{n^2} \notag
\end{equation}
\quad
Therefore, theoretically  there  exists  an  appropriate  $t$,  that  makes 
\begin{align}
& \frac{a_n}{n^2}=\frac{a_n+t\ell_n}{n(n+1)}<\frac{a_n+\ell_n}{n(n+1)},
\quad then \notag \\
& \frac{a_n}{n^2}<\frac{a_n+\ell_n}{n(n+1)}    \label{eq20}
\end{align}
\quad Does inequality \eqref{eq20} hold? Proceeding with a proof by contradiction, let us assume that \\
\begin{equation}
\frac{a_n}{n^2}\geq \frac{a_n+\ell_n}{n(n+1)} \label{eq21}
\end{equation}
\quad \ Then we can obtain
\begin{equation}
\ell_n \leq \frac{a_n}{n} =  \frac{\varphi(n^2,M_r)}{n} \label{eq22} 
\end{equation}

According to inequality \eqref{eq22}, similarly, we find that
\begin{equation}
\ell_{n-1} \leq \frac{\varphi((n-1)^2,M_r)}{n-1}
\end{equation}
\ \quad Therefore, 
\begin{align}
\varphi(n^2,M_r)=&\varphi((n-1)^2,M_r)+\ell_{n-1}  \notag \\
\leq &\varphi((n-1)^2,M_r)+\frac{\varphi((n-1)^2,M_r)}{n-1} 
=\frac{n}{n-1}\cdot \varphi((n-1)^2,M_r) \notag 
\end{align}

\quad That is, 
\begin{equation}
\varphi(n^2,M_r)\leq \frac{n}{n-1}\cdot \varphi((n-1)^2,M_r) \label{eq24}
\end{equation}

\quad Inequality \eqref{eq24} is the general formula; and we must pay attention to the change when crossing sections of adjacent prime numbers.
\begin{align}
\varphi(n^2,M_r)&\leq \frac{n}{n-1}\frac{n-1}{n-2}\dots \frac{p_r+1}{p_r}\varphi(p_r^2,M_r)
=\frac{n}{p_r}\varphi(p_r^2,M_r) \\
&\leq \frac{n}{p_r}\frac{p_r}{p_r-1}\varphi((p_r-1)^2,M_r)  \notag
\end{align}
\ \quad According to equation \eqref{eq11},
\begin{align}
\varphi(n^2,M_r)&\leq \frac{n}{p_r-1}\{\varphi((p_r-1)^2,M_{r-1})-1\}  \\
&=\frac{n}{p_r-1}\varphi((p_r-1)^2,M_{r-1})-\frac{n}{p_r-1} \notag \\
&\leq \frac{n}{p_r-1}\frac{p_r-1}{p_r-2}\varphi((p_r-1)^2,M_{r-1})-\frac{n}{p_r-1} \notag  \\
&\leq \frac{n}{p_r-1}\frac{p_r-1}{p_r-2}\dots \frac{p_{r-1}}{p_{r-1}-1}\{\varphi((p_{r-1}-1)^2,M_{r-2})-1\}-\frac{n}{p_r-1} \notag  \\
&=\frac{n}{p_{r-1}-1}\{\varphi((p_{r-1}-1)^2,M_{r-2})-1\}-\frac{n}{p_r-1} \notag \\
&=\frac{n}{p_{r-1}-1}\varphi((p_{r-1}-1)^2,M_{r-2})-\frac{n}{p_{r-1}-1}-\frac{n}{p_r-1} \notag  \\
&\dots \dots \notag \notag \\
&\leq \frac{n}{p_3-1}\varphi((p_3-1)^2,M_2)-\frac{n}{p_3-1}-\dots -\frac{n}{p_{r-1}-1}-\frac{n}{p_r-1}\notag  \\
\leq &\frac{n}{p_3-1}\frac{p_3-1}{p_3-2}\varphi((p_3-2)^2,M_2)-\frac{n}{p_3-1}-\dots -\frac{n}{p_{r-1}-1}-\frac{n}{p_r-1} \notag  \\
= &\frac{n}{p_3-2}\varphi((p_3-2)^2,M_2)-\frac{n}{p_3-1}-\dots -\frac{n}{p_{r-1}-1}-\frac{n}{p_r-1} \notag 
\end{align}
\qquad Because $p_1=2,p_2=3,p_3=5,M_1=2,M_2=6,$ 

\begin{align}
&\varphi(n^2,M_r)\leq \frac{n}{p_2}\varphi(p_2^2,M_2)-\frac{n}{p_3-1}-\dots -\frac{n}{p_{r-1}-1}-\frac{n}{p_r-1}\\
&\leq \frac{n}{p_2}\frac{p_2}{p_2-1}\varphi((p_2-1)^2,M_2)-\frac{n}{p_3-1}-\dots -\frac{n}{p_{r-1}-1}-\frac{n}{p_r-1} \notag  \\
&= \frac{n}{p_2-1}\{\varphi((p_2-1)^2,M_1)-1\}-\frac{n}{p_3-1}-\dots -\frac{n}{p_{r-1}-1}-\frac{n}{p_r-1}\notag   \\
&= \frac{n}{p_2-1}\varphi((p_2-1)^2,M_1)-\frac{n}{p_2-1}-\frac{n}{p_3-1}-\dots -\frac{n}{p_{r-1}-1}-\frac{n}{p_r-1} \notag  \\
&= \frac{n}{p_1}\varphi(p_1^2,M_1)-\frac{n}{p_2-1}-\frac{n}{p_3-1}-\dots -\frac{n}{p_{r-1}-1}-\frac{n}{p_r-1} \notag  \\
&= \frac{n}{2}\varphi(2^2,2)-\frac{n}{2}-\frac{n}{4}-\dots -\frac{n}{p_{r-1}-1}-\frac{n}{p_r-1} \notag  \\
&= n-\frac{n}{2}-\frac{n}{4}-\dots -\frac{n}{p_{r-1}-1}-\frac{n}{p_r-1} \notag 
\end{align}
\ \qquad That is ,
\begin{equation}
\quad \varphi(n^2,M_r)\leq n-\frac{n}{2}-\frac{n}{4}-\dots -\frac{n}{p_{r-1}-1}-\frac{n}{p_r-1} \label{eqan}
\end{equation}
\ \qquad When $p_r\geq 11(i.e \ n \geq 10)$,
\begin{equation}
\varphi(n^2,M_r)\leq n-\frac{n}{2}-\frac{n}{4}-\frac{n}{6}-\frac{n}{10}<0 \label{eqnn}
\end{equation}

\quad Because there is at least one integer 1 between $1\sim n^2$ that is coprime with $M_r$, expressed as $\varphi(n^2,M_r)\geq1$, the inequality \eqref{eqnn} does not hold. That is, when $n\geq 10$, inequality $\eqref{eqan}$ does not hold. Therefore,  the assumption is false. Therefore, when $n\geq 10$, \par
\begin{align}
&\frac{a_n}{n^2}<\frac{a_n+\ell_n}{n(n+1)}=\frac{a_n+\ell_n}{(n+1)^2}\cdot\frac{n+1}{n} \label{eq45} \\
{\therefore} \quad &\frac{\varphi(n^2,M_r)}{n^2}<\frac{\varphi((n+1)^2,M_r)}{(n+1)^2}\cdot\frac{n+1}{n}\notag \\
{\therefore} \quad &\frac{\varphi((n+1)^2,M_r)}{(n+1)^2}>\frac{\varphi(n^2,M_r)}{n^2}\cdot\frac{n}{n+1}\notag \\
{\therefore} \quad &f((n+1)^2,M_r)>f(n^2,M_r)\cdot\frac{n}{n+1} %\label{eqfn} 
\end{align}
\quad Hence, the correlation of the proportions of coprime numbers associated with consecutive square numbers presented in Theorem $\ref{theo1}$ is proven. This is called the general correlation model.
\end{proof}

\begin{table}[h]
\caption{Comparison between $f((n+1)^2,M_r)$ and $f(n^2,M_r)\cdot\frac{n}{n+1}$}\label{tab1}%
\begin{threeparttable}
\begin{tabular}{@{}llllllll@{}}
\toprule
n & n+1  & r & $p_r$ & $M_r$ & $f(n^2,M_r)$ & $f(n^2,M_r)\cdot\frac{n}{n+1}$ & $f((n+1)^2,M_r)$ \\
\midrule
 1  & 2 & 1  & 2  & 2  &     &   &   \\
 2  & 3 & 2  & 3  & 6  &  0.25000000  & 0.16666667 & 0.33333333 \\
 3  & 4 & 2  & 3  & 6  &  0.33333333  & 0.25000000 & 0.31250000 \\
 4  & 5 & 3  & 5  & 30  &  0.25000000  & 0.20000000 & 0.28000000 \\
 5  & 6 & 3  & 5  & 30  &  0.28000000  & 0.23333333 & 0.25000000 \\
 6  & 7 & 4  & 7  & 210  &  0.22222222  & 0.19047619 & 0.24489796 \\
 7  & 8 & 4  & 7  & 210  &  0.24489796 & 0.21428571 & 0.23437500 \\
 8  & 9 & 4  & 7  & 210  &  0.23437500  & 0.20833333 & 0.23456790 \\
 9  & 10 & 4  & 7  & 210  &  0.23456790  & 0.21111111 & 0.22000000 \\
 10  & 11 & 5  & 11  & 2310  &  0.21000000  & 0.19090909 & 0.21487603 \\
 11  & 12 & 5  & 11  & 2310  &  0.21487603  & 0.19696970 & 0.20833333 \\
 12  & 13 & 6  & 13  & 30030  &  0.20138889  & 0.18589744 & 0.20118343 \\
 13  & 14 & 6  & 13  & 30030  &  0.20118343  & 0.18681319 & 0.19897959 \\
 14  & 15 & 6  & 13  & 30030  &  0.19897959  & 0.18571429 & 0.19111111 \\
 15  & 16 & 6  & 13  & 30030  &  0.19111111  & 0.17916667 & 0.19140625 \\
\bottomrule 
\end{tabular}
 \begin{tablenotes}
\footnotesize
\item Where $\  n \geq p_r-1$
\end{tablenotes}
\end{threeparttable}
\end{table}

For example, when $n=13, n+1=14, r=6, p_r=13, M_r=2\cdot 3\cdot 5\cdot 7\cdot 11\cdot 13=30030$,
\begin{align}
f(13^2, 30030) &= \frac{\varphi(13^2,30030)}{13^2 }=\frac{34}{169}=0.20118343 \notag\\
f(13^2, 30030) &\cdot \frac{13}{14} = 0.20118343\cdot\frac{13}{14}=0.18681319 \notag\\
f(14^2, 30030) &= \frac{\varphi(14^2,30030)}{14^2}=\frac{39}{196}=0.19897959 \notag
\end{align}
\quad Obviously, $f(14^2, 30030)>f(13^2, 30030) \cdot \frac{13}{14}$. This result is consistent with Theorem $\ref{theo1}$.

More cases are shown in Table \ref{tab1}, which shows that not only when $n>10$, but also when $n \leq 10$, the proportions of coprime numbers associated with consecutive square numbers conform to the law of Theorem $\ref{theo1}$.

\begin{mythm}\label{theo2}
There is at least one prime number between $n^2$ and $(n+1)^2$.
\end{mythm}
\begin{proof}
According to Theorem $\ref{theo1}$ and from equation \eqref{eqln}, we obtain
\begin{equation}
\ell_n>n\cdot f(n^2,M_r)=\frac{\varphi(n^2,M_r)}{n} \label{eq49}
\end{equation}
\qquad (Inequality \eqref{eq49} can also be deduced from Inequality \eqref{eq45}.) 
\begin{align}
\because \quad &\varphi(n^2,M_r)\geq 1,\ n\geq 1 \notag \\
\therefore \quad &\ell_n>0 \notag  \\
\because \quad &\ell_n \  is\  an\  integer \notag \\
\therefore \quad &\ell_n\geq 1 \notag 
\end{align}

That is, there is at least one prime number between $n^2$ and ${(n+1)}^2$ when $p_{r+1}-1>n \geq p_r-1$. Thus, we have proven Legendre's conjecture. Some cases are shown in the following table (Table \ref{tab2}).\\
\end{proof}

\begin{table}[h]
\caption{Comparison between $\ell_n$ and $nf(n^2,M_r)$}\label{tab2}
\begin{threeparttable}
\begin{tabular}{@{}llllllll@{}}
\toprule
n & n+1  & r & $p_r$ & $M_r$ & $f(n^2,M_r)$ & $nf(n^2,M_r)$ & $\ell_n$ \\
\midrule
 1  & 2 & 1  & 2  & 2  &     &   &   \\
 2  & 3 & 2  & 3  & 6  &  0.25000000  & 0.50 & 2 \\
 3  & 4 & 2  & 3  & 6  &  0.33333333  & 1.00 & 2 \\
 4  & 5 & 3  & 5  & 30  &  0.25000000  & 1.00 & 3 \\
 5  & 6 & 3  & 5  & 30  &  0.28000000  & 1.40 & 2 \\
 6  & 7 & 4  & 7  & 210  &  0.22222222  & 1.33 & 4 \\
 7  & 8 & 4  & 7  & 210  &  0.24489796  & 1.71 & 3 \\
 8  & 9 & 4  & 7  & 210  &  0.23437500  & 1.88 & 4 \\
 9  & 10 & 4  & 7  & 210  &  0.23456790  & 2.11 & 3 \\
 10  & 11 & 5  & 11  & 2310  &  0.21000000  & 2.10 & 5 \\
 11  & 12 & 5  & 11  & 2310  &  0.21487603  & 2.36 & 4 \\
 12  & 13 & 6  & 13  & 30030  &  0.20138889  & 2.42 & 5 \\ 
 13  & 14 & 6  & 13  & 30030  &  0.20118343  & 2.62 & 5 \\ 
 14  & 15 & 6  & 13  & 30030  &  0.19897959  & 2.79 & 4 \\  
 15  & 16 & 6  & 13  & 30030  &  0.19111111  & 2.87 & 6 \\  
 16  & 17 & 7  & 17  & 510510  &  0.18750000  & 3.00 & 7 \\  
 17  & 18 & 7  & 17  & 510510  &  0.19031142  & 3.24 & 5 \\ 
 18  & 19 & 8  & 19  & 9699690  &  0.18209877  & 3.28 & 6 \\ 
 19  & 20 & 8  & 19  & 9699690  &  0.18005540  & 3.42 & 6 \\ 
 20  & 21 & 8  & 19  & 9699690  &  0.17750000  & 3.55 & 7 \\ 
\bottomrule 
\end{tabular}
 \begin{tablenotes}
\footnotesize
\item Where $\  n \geq p_r-1$
\end{tablenotes}
\end{threeparttable}
\end{table}

\section{\textbf{Further research on prime numbers between $n^2$ and $(n+1)^2$}}\label{sec4}
In this section, we use the Bertrand–Chebyshev theorem to successfully establish the relationship between the backwards coprime number ratios of consecutive square numbers. In Theorem $\ref{theo4}$, we further derive that the relationship between the proportion of the coprimes in the interval of $n^2 \sim (n+1)^2$ and the proportions of the backwards coprimes associated with consecutive square numbers. In Theorem $\ref{theo5}$, we combine the relationship between the forward and backwards coprime number ratios of consecutive square numbers to derive the relationship between the proportion of coprines in the interval of of $n^2 \sim (n+1)^2$ and the proportion of coprimes in the entire interval. Then, we conclude that there are at least two prime numbers between consecutive square numbers.

\begin{mythm}\label{theo3}
The relationship between the proportions of backwards coprime numbers associated with consecutive square numbers is
\begin{equation}
 f(\overline{n^2},M_r)>f(\overline{(n+1)^2},M_r)\cdot \frac{\sqrt{{M_r}^2-(n+1)^2}}{{M_r}^2-n^2}  \label{eq52}
\end{equation}
\end{mythm}
\begin{proof}
The number of coprimes between consecutive square numbers is equal to both the difference in the number of coprimes from $1$ to each of these consecutive square numbers and the difference in the number of coprimes from $M_r^2$ to each of these consecutive square numbers. We can derive from inequality \eqref{eqfn} that there is at least one prime number between any consecutive square numbers. In other words, the property of "there is at least one prime number between any two consecutive square numbers" follows the relation between the proportions of coprime factors (that are coprime with $M_r$ from $1$ to two consecutive square numbers (in the forwards direction)) that is expressed in inequality \eqref{eqfn}. Here, we can also conjecture that the proportions of coprime factors (those that are coprime with $M_r$ from $M_r^2$ to two consecutive square numbers (in the backwards direction)) follow a similar relationship to inequality \eqref{eqfn}, that is, formula \eqref{eq52}.\par
Here, we use the method of proof by contradiction. If we assume that formula \eqref{eq52} does not hold, then,
\begin{equation}
 f(\overline{n^2},M_r)\leq f(\overline{(n+1)^2},M_r)\cdot \frac{\sqrt{{M_r}^2-(n+1)^2}}{{M_r}^2-n^2}
\end{equation}
That is
\begin{equation}
\frac{M_r\varphi(M_r)-\varphi(n^2,M_r\ )}{M_r^2-n^2}\le\frac{M_r\varphi(M_r)-\varphi\left((n+1)^2,M_r\right)}{M_r^2-{(n+1)}^2}\cdot\frac{\sqrt{M_r^2-(n+1)^2}}{\sqrt{M_r^2-n^2}} \notag
\end{equation}
\begin{equation}
M_r\varphi(M_r)-\varphi(n^2,M_r)\le\frac{\sqrt{M_r^2-n^2}}{\sqrt{M_r^2-(n+1)^2}}\cdot[M_r\varphi(M_r)-\varphi\left((n+1)^2,M_r\right)] \label{eq55}
\end{equation}
\qquad For simplicity,  let 
\begin{equation}
b_n=M_r\varphi(M_r)-\varphi(n^2 ,M_r) \notag
\end{equation}
\qquad Similarly,

\begin{equation}
b_{n+1}=M_r\varphi(M_r)-\varphi((n+1)^2,M_r)  \notag
\end{equation}

\quad From inequality\eqref{eq55},
\begin{equation}
b_n\le\frac{\sqrt{M_r^2-n^2}}{\sqrt{M_r^2-(n+1)^2}}\cdot b_{n+1}
\end{equation}
\quad The above formula is the general formula, from which we can find that
\begin{align}
b_n&\le\frac{\sqrt{M_r^2-n^2}}{\sqrt{M_r^2-(n+1)^2}}\cdot\frac{\sqrt{M_r^2-(n+1)^2}}{\sqrt{M_r^2-(n+2)^2}}\cdot b_{n+2}  \\
 &\le\frac{\sqrt{M_r^2-n^2}}{\sqrt{M_r^2-(n+1)^2}}\cdot\frac{\sqrt{M_r^2-(n+1)^2}}{\sqrt{M_r^2-(n+2)^2}}\cdot\frac{\sqrt{M_r^2-(n+2)^2}}{\sqrt{M_r^2-(n+3)^2}}\cdots\notag \notag \\
&\  \cdot\frac{\sqrt{M_r^2-(M_r-2)^2}}{\sqrt{M_r^2-(M_r-1)^2}}\cdot b_{M_r-1} \notag \\
&=\frac{\sqrt{M_r^2-n^2}}{\sqrt{M_r^2-({M_r-1)}^2}}\cdot b_{M_r-1} \notag \\
& =\frac{\sqrt{M_r^2-n^2}}{\sqrt{2M_r-1}}\cdot b_{M_r-1} \notag 
\end{align}
\begin{align}
\because b_{M_r-1}&=\varphi(M_r^2,M_r)-\varphi((M_r-1)^2,M_r) \notag \\
&=M_r\varphi(M_r)-\varphi((M_r-1)^2,M_r) \notag \\
\because (M_r-&1)^2={M_r}^2-2M_r+1=M_r(M_r-2)+1 \notag \\
\therefore \varphi((M_r&-1)^2,M_r) \geq \varphi(M_r(M_r-2),M_r)=(M_r-2)\varphi(M_r) \notag \\
\therefore b_{M_r-1}&\le 2\varphi(M_r)\\ 
\therefore \ b_n<&\frac{\sqrt{M_r^2-n^2}}{\sqrt{2M_r-1}}\cdot 2\varphi(M_r)<\frac{\sqrt{M_r^2}}{\sqrt{M_r}}\cdot 2\varphi(M_r)<2\sqrt{M_r}\varphi(M_r) \label{eq67} 
\end{align}
\qquad Because $M_r=p_1p_2p_3\cdots p_{r-1}p_r=2\cdot 3\cdot 5\cdots p_{r-1}p_r$, when $p_{r+1}-1>n\geq p_r-1$, according to the Bertrand–Chebyshev theorem, there exists at least one prime number between $n$ and $2n$, we have
\begin{align}
&p_1p_r>p_{r+1}>n \notag \\
& p_3p_{r-1}>\frac{5}{2}p_r>\frac{5}{4}p_{r+1}>n \notag 
\end{align}
\qquad Therefore, when $r\geq5$, 
\begin{align}
M_r&=p_1p_2p_3\cdots p_{r-1}p_r>{3n}^2 \notag \\
\therefore \ &n^2<\frac{M_r}{3} \notag \\
\therefore \ &\varphi(n^2,M_r)<\frac{n^2}{2}-r+1<\frac{n^2}{2}<\frac{M_r}{6}\\
\therefore \ &M_r\varphi(M_r)-\varphi(n^2,M_r) >M_r\varphi(M_r)-\frac{M_r}{6}>\frac{5}{6}M_r\varphi(M_r)
\end{align}
\quad Obviously, when $r\geq 5$, that is, $M_r\geq2\cdot 3\cdot 5\cdot 7\cdot 11$, the following holds:
\begin{align}
\frac{5}{6} M_r \varphi(M_r)&>2\sqrt{M_r}\varphi(M_r) \\
\therefore \  b_n=M_r&\varphi(M_r)-\varphi(n^2,M_r)>2 \sqrt{M_r}\varphi(M_r) \label{eq76}
\end{align}
\ \quad Formula \eqref{eq76} contradicts formula \eqref{eq67} when $n\geq 10 (i.e \ p_r\geq11$). Therefore, formula \eqref{eq67} derived from the previous assumption is not valid, implying the assumption itself is not valid. Thus, inequality \eqref{eq52} holds when $n\geq10$. \\
\end{proof}
As shown in the following table (Table \ref{tab3}), not only when $n\geq 10$, but also actually when $n>1$, the proportion of backwards prime numbers associated with consecutive square numbers conforms to the law of Theorem \ref{theo3}.\\
\begin{table}[h]
\caption{Comparison between  $f(\overline{n^2},M_r)$ and $f(\overline{(n+1)^2},M_r)\cdot \frac{\sqrt{{M_r}^2-(n+1)^2}}{{M_r}^2-n^2}$} \label{tab3}
\begin{threeparttable}
\begin{tabular}{@{}llllllll@{}}
\toprule
n & n+1  & r & $p_r$ & $M_r$ & $f(\overline{(n+1)^2},M_r)$ & $f(\overline{(n+1)^2},M_r)\cdot \frac{\sqrt{{M_r}^2-(n+1)^2}}{{M_r}^2-n^2}$ & $f(\overline{n^2},M_r)$ \\
\midrule
1  & 2 & 1  & 2  & 2  &    &  &  \\
2 & 3 & 2 & 3 & 6 & 0.333333333333 & 0.306186217848 & 0.343750000000 \\
3 & 4 & 2 &3 & 6 & 0.350000000000 & 0.301232038038 & 0.333333333333 \\
4 & 5 & 3 & 5 & 30 & 0.266285714286 & 0.264926719631 & 0.266968325792 \\
5 & 6 & 3 & 5 & 30 & 0.267361111111 & 0.265675240347 & 0.266285714286 \\
6 & 7 & 4 & 7 & 210 & 	0.228553267803 & 0.228519550805 & 0.228576615832 \\
7 & 8 & 4 & 7 & 210 &  0.228562993914 & 0.228524076105 & 0.228553267803 \\
8 & 9 & 4 & 7 & 210 & 0.228560394375 & 0.228516272500 & 0.228562993914 \\
9 & 10 & 4 & 7 & 210 & 0.228590909091 & 0.228541570214 & 0.228560394375 \\
10 & 11 & 5 & 11 & 2310 & 0.207792047158 & 0.207791638271 & 0.207792166417 \\
11 & 12 & 5 & 11 & 2310 & 0.207792193189 & 0.207791745359 & 0.207792047158 \\
12 & 13 & 6 & 13 & 30030 & 0.191808190051 & 0.191808187393 & 0.191808190278 \\
13 & 14 & 6 & 13 & 30030 & 0.191808190250 & 0.191808187378 & 0.191808190051 \\
14 & 15 & 6 & 13 & 30030 & 0.191808191982 & 0.191808188898 & 0.191808190250 \\
15 & 16 & 6 & 13 & 30030 & 0.191808191922 & 0.191808188626 & 0.191808191982 \\
16 & 17 & 7 & 17 & 510510 & 0.180525356985 & 0.180525356974 & 0.180525356989 \\
17 & 18 & 7 & 17 & 510510 & 0.180525356990 & 0.180525356978 & 0.180525356985 \\
18 & 19 & 8 & 19 & 9699690 & 0.171024022417177 & 0.171024022417143 & 0.171024022417173 \\
19 & 20 & 8 & 19 & 9699690 & 0.171024022417184 & 0.171024022417148 & 0.171024022417177 \\
20 & 21 & 8 & 19 & 9699690 & 0.171024022417184 & 0.171024022417147 & 0.171024022417184 \\
\bottomrule 
\end{tabular}
 \begin{tablenotes}
\footnotesize
\item Where $\  n \geq p_r-1$
\end{tablenotes}
\end{threeparttable}
\end{table}

\begin{mythm}\label{theo4}
The relationship between the proportion of coprime numbers in the interval of $n^2 \sim (n+1)^2$ and the proportions of backwards coprime numbers associated with these consecutive numbers is as follows:
\begin{equation}
\frac{\ell_n}{n}>\frac{M_r \varphi(M_r)-\varphi(n^2,M_r)}{{M_r}^2-n^2} 
\end{equation}
\end{mythm}
\begin{proof}
According to Theorem \ref{theo3},
\begin{small}
\begin{align}
&\frac{M_r\varphi(M_r)-\varphi(n^2,M_r)}{M_r^2-n^2}>\frac{M_r\varphi(M_r)-\varphi\left((n+1)^2,M_r\right)}{M_r^2-{(n+1)}^2}\cdot\frac{\sqrt{M_r^2-(n+1)^2}}{\sqrt{M_r^2-n^2}} \notag \\
&\therefore \  \frac{M_r\varphi(M_r)-\varphi(n^2,M_r)}{\sqrt{M_r^2-n^2}}>\frac{M_r\varphi(M_r)-\varphi\left((n+1)^2,M_r\right)}{\sqrt{M_r^2-(n+1)^2}} \label{eq78}
\end{align}
\begin{equation}
\therefore \  \frac{[ M_r\varphi(M_r)-\varphi(n^2,M_r)] -[ M_r\varphi(M_r)-\varphi\left((n+1)^2,M_r\right)] }{\sqrt{M_r^2-n^2} -\sqrt{M_r^2-(n+1)^2}}>\frac{M_r\varphi(M_r)-\varphi(n^2 ,M_r)}{\sqrt{M_r^2-n^2}} \notag
\end{equation}
\vskip2mm
\begin{align}
&\therefore \  \frac{\varphi\left((n+1)^2,M_r\right) -\varphi(n^2,M_r)}{\sqrt{M_r^2-n^2} -\sqrt{M_r^2-(n+1)^2}}>\frac{M_r\varphi(M_r)-\varphi(n^2 ,M_r)}{\sqrt{M_r^2-n^2}} \notag \\
&\therefore \  \frac{\ell_n}{\sqrt{M_r^2-n^2} -\sqrt{M_r^2-(n+1)^2}}>\frac{M_r\varphi(M_r)-\varphi(n^2 ,M_r)}{\sqrt{M_r^2-n^2}}\\
&\because \sqrt{M_r^2-n^2}=\notag\\
&\quad \sqrt{M_r^2-\left(n+1\right)^2+2n+\left[\frac{n}{\sqrt{M_r^2-\left(n+1\right)^2}}\right]^2+1-\left[\frac{n}{\sqrt{M_r^2-\left(n+1\right)^2}}\right]^2}\notag\\
&\quad >\sqrt{M_r^2-{(n+1)}^2}+\frac{n}{\sqrt{M_r^2-\left(n+1\right)^2}} \notag \\
&\therefore \  \sqrt{M_r^2-n^2}-\sqrt{M_r^2-{(n+1)}^2}>\frac{n}{\sqrt{M_r^2-\left(n+1\right)^2}} \notag \\
&\therefore \  \frac{1}{\frac{n}{\sqrt{M_r^2-\left(n+1\right)^2}}}>\frac{1}{\sqrt{M_r^2-n^2}-\sqrt{M_r^2-{(n+1)}^2}} \notag \\
&\therefore \  \frac{\ell_n}{\frac{n}{\sqrt{M_r^2-\left(n+1\right)^2}}}>\frac{\ell_n}{\sqrt{M_r^2-n^2}-\sqrt{M_r^2-{(n+1)}^2}}>\frac{M_r\varphi(M_r)-\varphi(n^2,M_r)}{\ \sqrt{M_r^2-n^2}} \label{eq85} \\
&\therefore \  \frac{\ell_n}{n}>\frac{M_r\varphi(M_r)-\varphi(n^2,M_r)}{\sqrt{M_r^2-\left(n+1\right)^2}\sqrt{M_r^2-n^2}} >\frac{M_r\varphi(M_r)-\varphi(n^2,M_r)}{M_r^2-n^2}=f(\overline{n^2},Mr)
\end{align}
\end{small}
\vskip2mm

\end{proof}

\quad Similarly, from inequality \eqref{eq78} and \eqref{eq85}, we can also derive that

\begin{equation}
\frac{\ell_n}{n}>\frac{M_r\varphi(M_r)-\varphi((n+1)^2,M_r)}{M_r^2-{(n+1)}^2}=f(\overline{(n+1)^2},Mr)
\end{equation}

As shown in the following table (Table \ref{tab4}), when $n>1$, the relationship between the proportion of coprime numbers in the interval of $n^2 \sim (n+1)^2$ and the proportion of backwards coprime numbers follows the law of Theorem \ref{theo4}.

\begin{table}[h]
\caption{Comparison between $\frac{\ell_n}{n}$ and $f(\overline{n^2},Mr)$}\label{tab4}
\begin{threeparttable}
\begin{tabular}{@{}lllllll@{}}
\toprule
n & n+1  & r & $p_r$ & $M_r$ & $f(\overline{n^2},Mr)$ & $\frac{\ell_n}{n}$ \\
\midrule
 1  & 2 & 1  & 2  & 2  &     &   \\
 2  & 3 & 2  & 3  & 6  &  0.343750000000  & 1.0000 \\
 3  & 4 & 2  & 3  & 6  &  0.333333333333  & 0.6667 \\
 4  & 5 & 3  & 5  & 30  &  0.266968325792  & 0.7500 \\
 5  & 6 & 3  & 5  & 30  &  0.266285714286  & 0.4000 \\
 6  & 7 & 4  & 7  & 210  &  0.228576615832  & 0.6667 \\
 7  & 8 & 4  & 7  & 210  &  0.228553267803  & 0.4286 \\
 8  & 9 & 4  & 7  & 210  &  0.228562993914  & 0.5000 \\
 9  & 10 & 4  & 7  & 210  &  0.228560394375  & 0.3333 \\
 10  & 11 & 5  & 11  & 2310  &  0.207792166417  & 0.5000 \\
 11  & 12 & 5  & 11  & 2310  &  0.207792047158  & 0.3636 \\
 12 & 13 & 6 & 13 & 30030 & 0.191808190278 & 	0.4167 \\
 13 & 14 & 6 & 13 & 30030 & 0.191808190051 & 	0.3846 \\
 14 & 15 & 6 & 13 & 30030 & 0.191808190250 & 	0.2857 \\
 15 & 16 & 6 & 13 & 30030 & 0.191808191982 & 	0.4000 \\
 16 & 17 & 7 & 17 & 510510 & 0.180525356989 & 0.4375 \\
 17 & 18 & 7 & 17 & 510510 & 0.180525356985 & 0.2941 \\
 18 & 19 & 8 & 19 & 9699690 & 0.171024022417173 & 0.3333 \\
 19 & 20 & 8 & 19 & 9699690 & 0.171024022417177 & 0.3158 \\
 20 & 21 & 8 & 19 & 9699690 & 0.171024022417184 & 0.3500 \\ 
 \bottomrule 
\end{tabular}
 \begin{tablenotes}
\footnotesize
\item Where $\  n \geq p_r-1$
\end{tablenotes}
\end{threeparttable}
\end{table}

\begin{mythm}\label{theo5}
There are at least 2 prime numbers between $n^2$ and $(n+1)^2$.
\end{mythm}
\begin{proof}
According to Theorem \ref{theo2} and \ref{theo4}, 
\begin{align}
\frac{\ell_n}{n}&>\frac{\varphi(n^2,M_r)}{n^2} \qquad (n\geq p_r-1) \notag\\
\frac{\ell_n}{n}&>\frac{M_r\varphi(M_r)-\varphi(n^2\ ,M_r)}{M_r^2-n^2} \qquad (n\geq p_r-1) \notag
\end{align}
\qquad Therefore, when $n\geq p_r-1$
\begin{align}
\frac{\ell_n}{n}>&\frac{M_r\varphi\left(M_r\right)-\varphi\left(n^2,M_r\right)+\varphi\left(n^2,M_r\right)}{M_r^2-n^2+n^2}
=\frac{M_r\varphi(M_r)}{M_r^2}=\frac{\varphi(M_r)}{M_r}=f(M_r) \notag \\
\therefore \  &\ell_n>n\cdot f(M_r)=n\prod_{i=1}^{r}{(1-\frac{1}{p_i}})
\end{align}

\qquad Then, 

\begin{align}
\ell_n&>(p_r-1)\cdot\prod_{i=1}^{r}{(1-\frac{1}{p_i}})\\
&=\frac{p_1-1}{p_1}\cdot \frac{p_2-1}{p_2}\cdot \frac{p_3-1}{p_3}\cdots\frac{p_{r-1}-1}{p_{r-1}}\cdot \frac{p_r-1}{p_r}\cdot (p_r-1) \notag \\
&=\frac{1}{2}\cdot \frac{2}{3}\cdot \frac{4}{5}\cdots\frac{p_{r-1}-1}{p_{r-1}}\cdot \frac{(p_r-1)^2}{p_r} \notag \\
&>\frac{1}{2}\cdot \frac{2}{3}\cdot \frac{3}{4}\cdots \frac{r-1}{r}\cdot \frac{(p_r-1)^2}{p_r} \notag \\
&=\frac{(p_r-1)^2}{r\cdot p_r} \notag \\
&>\frac{p_r-2}{r}  \notag   
\end{align}
\qquad \quad That  is,  
\begin{equation}
\ell_n>n\prod_{i=1}^{r}{(1-\frac{1}{p_i}})>\frac{p_r-2}{r}
\end{equation}

\qquad When $r \geq 5$, $\frac{p_r-2}{r}>1$, so $\ell_n>1$. \par
\vskip2mm
\qquad Because $\ell_n$ is an integer, $\ell_n\geq 2$.\par
\vskip2mm
When $p_{r+1}-1>n\geq p_r-1$, $\ell_n$ is both the number of coprimes between $n^2$ and $(n+1)^2$ that are coprime with $M_r$, and the number of prime numbers between $n^2$ and $(n+1)^2$. Thus, there are at least 2 prime numbers between $n^2$ and ${(n+1)}^2$.\\
\end{proof}

\section{\textbf{Conclusions}}\label{sec5}
In this article, we use elementary mathematical methods to expand the application of Euler's function, combined with the Bertrand–Chebyshev theorem, and we propose an innovative method for analysing the proportion of coprime numbers. We believe that this study proves that there are at least two prime numbers between $n^2$ and $(n+1)^2$ to a certain extent, and also reveals some rules about consecutive square numbers and their associated coprime numbers. We welcome criticism and suggestions from colleagues and hope that this work is helpful in inspiring researchers to further explore the secrets of prime numbers.

\end{document}